\documentclass[11pt]{article}
\usepackage[margin=1in]{geometry}
\usepackage{amsthm,amsmath}
\usepackage{latexsym,amssymb,amsmath}
\usepackage{graphicx,float,color,fancybox,shapepar,setspace,hyperref}
\usepackage{subfigure}
\usepackage{pgf,tikz}
\usepackage[affil-it]{authblk}
\usepackage{indentfirst}
\usepackage{hyperref}
\usepackage[section]{placeins}
\usepackage{todonotes}
\hypersetup
{
colorlinks=true,
linkcolor=blue,
filecolor=blue,
urlcolor=blue,
citecolor=cyan,
}
\usetikzlibrary{arrows}
\usepackage{graphicx}
\usepackage{amsfonts}
\usepackage{amssymb}
\usepackage{amsmath}
\usepackage{amsthm}
\usepackage[english]{babel}
\usepackage{ctex}
\usepackage{hyperref}
\usepackage{cases}
\usepackage{authblk}
\usepackage{mathtools}
\usepackage{pifont}
\usepackage[numbers,sort&compress]{natbib}

\newtheorem{lemma}{Lemma}[section]
\newtheorem{theorem}[lemma]{Theorem}

\begin{document}
	
	\title{The hat guessing number of random graphs with constant edge-chosen probability}

\author{Lanchao Wang,  Yaojun Chen\thanks{Corresponding author. 
 Email: 
\href{mailto://yaojunc@nju.edu.cn}{ yaojunc@nju.edu.cn}.
}}
 \affil{ { \small {Department of Mathematics, Nanjing University, Nanjing 210093, China}}
 }

\date{\today}
\maketitle
\begin{abstract}

Let $G$ be a graph with $n$ vertices.
The {\em hat guessing number} of $G$ is defined in terms of the following game:
There are $n$  players and one opponent.
The opponent will wear one of the $q$ hats of different colors on the player's head.
At this time, the player can only see the player's hat color at the adjacent vertex,
and communication between players is not allowed.
Once players are assigned hats, each player must guess the color of his hat at the same time.
If at least one player guesses right, then they will win collectively.
Given a graph $G$, its hat guessing number $HG(G)$ is the largest integer $q$ such that
there exists a guessing strategy guaranteeing at least one correct guess for any hat
assignment of $q$ different colors.

Let $\mathcal{G}(n,p)$ denote the  Erd\H{o}s-R\'{e}nyi random graphs
with $n$ vertices and edge-chosen probability $ p\in(0,1)$.
Alon-Chizewer and Bosek-Dudek-Farnik-Grytczuk-Mazur investigated the
lower and upper bound for $HG(G)$ when $G\in \mathcal{G}(n,1/2)$, respectively.
In this paper, we extends their results by showing that for any constant
number $p$, we have $n^{1 - o(1)}\le HG(G) \le (1-o(1))n$ with high probability when $G\in \mathcal{G}(n,p)$.

\vskip 2mm
\noindent{\bf 2020 Mathematics Subject Classification}: 05C57, 05C80

\noindent{\bf Keywords}: Hat guessing number; Random graph

	\end{abstract}

	\section{Introduction}

\vskip.10cm

The {\em hat guessing number} of a graph $G$ on $n$ vertices was defined by Bosek,  Dudek,  Farnik,  Grytczuk and  Mazur \cite{butler2008} in terms of the
following game:
There are $n$  players and one opponent.
The opponent will wear one of the $q$ hats of different colors on the player's head.
At this time, the player can only see the player's hat color at the adjacent vertex,
and communication between players is not allowed.
Once players are assigned hats, each player must guess the color of his hat at the same time.
If at least one player guesses right, then they will win collectively.
Given a graph $G$, its hat guessing number $HG(G)$ is the largest integer $q$ such that
there exists a guessing strategy guaranteeing at least one correct guess for any hat
assignment of $q$ different colors.
 The value of $HG(G)$ has been studied in several papers
including \cite{alon2020,alon2022,bosek2021,butler2008,
gadou2018,gadou2015,he2022,he2020,kokh2019,lll,szcz2017}.

\vskip.1cm

Let $\mathcal{G}(n,p)$ denote the  Erd\H{o}s-R\'{e}nyi random graphs
with $n$ vertices and edge-chosen probability $ p\in(0,1)$.
A sequence of events $\{ \mathcal{A}_n \}$ is said to occur {\it with
high probability} if $\mathbb{P} (\mathcal{A}_n) \to 1$ as $n\to \infty$.
It is natural to estimate the hat guessing number of the random graphs.
 Bosek, Dudek, Farnik, Grytczuk and Mazur started to study  the hat guessing number of $G\in \mathcal{G}(n,1/2)$ and
proved that $(2-o(1))\log_2 n \leq HG(G) \leq n-(1 + o(1))\log_2 n$ with high probability.
Along their work,  Alon and Chizewer \cite{alon2022} improved the lower bound of $(2-o(1))\log_2 n$ considerably into $n^{1-o(1)}$.

In this paper, we study the hat guessing number of the random graph $G\in \mathcal{G}(n,p)$ with $p$ an arbitrarily constant,
and extend their results as follows.

{ \begin{theorem}\label{1}
Let $G$ be a graph chosen from $\mathcal{G}(n,p)$ and $p$ is a constant not
dependent on $n$. Then $n^{1 - o(1)}\le HG(G) \le (1-o(1))n$ with high probability.
\end{theorem}
}

\section{Proof~of~Theorem~\ref{1}}\label{HG1}

\vskip.1cm

Our main idea  comes from \cite{alon2022} and \cite{bosek2021}.
 We begin with some lemmas which are very useful in the proof of our result.

 The first lemma is well known, which gives an asymptotically  estimate of the chromatic
 number $\chi(G)$ and the clique number  $\omega(G)$ of a random graph $G\in \mathcal{G}(n,p)$.

\begin{lemma}\label{l1}\cite{B}
Let $G$ be a graph chosen from $\mathcal{G}(n,p)$. Then
we have $\chi(G) = (1+o(1))\frac{n}{2\log_{1/(1-p)} n}$ and $w(G)=(1-o(1))2\log_{1/p} n $ with high probability.
\end{lemma}

 The following lemma, due to  Bosek, Dudek, Farnik, Grytczuk and Mazur \cite{bosek2021}, gives an upper bound for $HG(G)$ in terms of $\chi(G)$.

\begin{lemma}\label{l3}\cite{bosek2021}
 Let $G=(V,E)$ be a graph of order $n$ with chromatic number $\chi(G)\ge 4$. Then,
	\[
	HG(G) \le \left(1- \frac{1}{2\chi(G)}\right) n
	\]
	for sufficiently large $n$.
\end{lemma}

{ For two positive integers $d$ and $n$,}
the {\em book graph} $B_{d,n}$ is obtained by adding $n$ nonadjacent common neighbors to the complete graph $K_d$. Note that if  a graph $H$ is a subgraph of $G$, then $HG(H)\le HG(G)$. So the following lemma, due to Alon and Chizewer \cite{alon2022}, will provide a lower bound for $HG(G)$ if we can show that $G$ contains  $B_{d,n}$ as a subgraph.

\begin{lemma}\cite{alon2022}\label{lemma}
For $n = d^{d+3}$, we have $ HG(B_{d,n}) \geq q = d^{d-2}$.
\end{lemma}

The following lemma is well known, used to bound the sum of independent random variables.

\begin{lemma}\label{l2}(Chernoff bound)
 Suppose that $X_1, X_2,\dots, X_n$ are independent random variables and
 that $a_i\le X_i\le b_i$ for $i=1,2,\dots,n$. Let $X=X_1+X_2+\cdots+X_n$ and $\mu=\mathbb{E}(X)$.
 Denote $c_i=b_i-a_i$, then for any $t>0$, we have
$$\mathbb{P}(X\le \mu-t)\le \exp \left(-\frac{2t^2}{c_1^2+c_2^2+\cdots+c_n^2}\right).$$
\end{lemma}

\vspace{0.2cm}
 The following asymptotically estimate will be useful in our calculations.
\begin{lemma}\label{l11}
Let $x$ be a real number satisfying $x^x=n$, then we have $x=(1+o(1))\log n/ \log \log n$.
\end{lemma}
\begin{proof}
Take the logarithm of both sides twice, we have $\log x+ \log \log x=\log \log n$ and
thus $\log x =(1+o(1)) \log \log n$. Note that $x\log x=\log n$, so $x=(1+o(1))\log n/ \log \log n$, as desired.
\end{proof}

Now, we are ready to begin the main proof.

\begin{proof}[\bfseries{Proof~of~Theorem~\ref{1}}]
For the upper bound, by combining Lemmas \ref{l1} and \ref{l3}, we have
	\begin{align*}
	HG(G) &\le \left(1- \frac{1}{2\chi(G)}\right) n \\
&\le \left(1- \frac{\log_{1/(1-p)} n}{(1+o(1))n}\right) n \\
&= n-(1-o(1))\log_{1/(1-p)} n=(1-o(1))n
	\end{align*}
 with high probability, as desired.

For the lower bound, by Lemmas \ref{l2} and \ref{l11}, with high
probability there exists a clique $K$ of order $d$ in $G$, where $d$ is
an integer satisfying $d^d \cdot d^3 \leq 0.5 np^{d}$ and $d = (1+o(1))\log n/\log \log n $.
Let $N$ denote the number of common neighbors of $K$, then
 $\mathbb{E}(N)=(n-d)p^d$. Applying Lemma \ref{l2} with $t=(0.5 n-d)p^d$, we have
	\begin{align*}
\mathbb{P}(N> 0.5 n p^{d})&\ge 1-\exp \left( -\frac{2 (0.5n-d)^2 p^{2d}}{n} \right)\\&
\ge 1- \exp\left( -0.4 n p^{2\log n/ \log\log n} \right)
\\
&=1-o(1).
	\end{align*}
 Thus $V(K)$ has more than $0.5 np^{d} \geq d^d \cdot d^3$
common neighbors with high probability.
This means that the book graph $B_{d,m}$ with $m = d^d \cdot d^3$ is a subgraph of $G$ with high probability.
So, by Lemmas \ref{lemma} and \ref{l11}, we have $HG(G)\ge d^{d-2}=n^{1-o(1)}$.
This completes the proof of Theorem \ref{1}.
\end{proof}

\vskip.1cm

\section*{Acknowledgments}
 
 
 This research was supported by NSFC under grant numbers  12161141003 and 11931006.   We would like
to thank Kaiyang Lan for many  helps.

\section*{Declaration}
\noindent$\textbf{Conflict~of~interest}$
The authors declare that they have no known competing financial interests or personal
relationships that could have appeared to influence the work reported in this paper.

\section*{Date availability}

No data was used for the research described in the article.

\end{document}